\documentclass[11pt,reqno,twoside]{amsart}

\usepackage{amsfonts}
\usepackage{amsmath,amssymb}
\usepackage[dvips,bottom=1.4in,right=1in,top=1in, left=1in]{geometry}
\usepackage{epic}
\usepackage{pstricks}
\usepackage{graphicx}

\newtheorem{theorem}{Theorem}[section]

\newtheorem{corollary}[theorem]{Corollary}
\newtheorem{definition}[theorem]{Definition}

\newtheorem{lemma}[theorem]{Lemma}

\newtheorem{proposition}[theorem]{Proposition}

\newtheorem{remark}[theorem]{Remark}

\makeatletter
\def\eqnarray{\stepcounter{equation}\let\@currentlabel=\theequation
\global\@eqnswtrue
\tabskip\@centering\let\\=\@eqncr
$$\halign to \displaywidth\bgroup\hfil\global\@eqcnt\z@
  $\displaystyle\tabskip\z@{##}$&\global\@eqcnt\@ne
  \hfil$\displaystyle{{}##{}}$\hfil
  &\global\@eqcnt\tw@ $\displaystyle{##}$\hfil
  \tabskip\@centering&\llap{##}\tabskip\z@\cr}

\def\endeqnarray{\@@eqncr\egroup
      \global\advance\c@equation\m@ne$$\global\@ignoretrue}

\def\@yeqncr{\@ifnextchar [{\@xeqncr}{\@xeqncr[5pt]}}
\makeatother

\DeclareMathAlphabet\gothic{U}{euf}{m}{n}

\newcommand{\gota}{\gothic{a}}

\def\R{\mathbb{R}}
\def\N{\mathbb{N}}
\def\C{\mathbb{C}}

\def\op{\operatorname{Re}}
\def\loc{\operatorname{loc}}

 \numberwithin{equation}{section}

\title[Asymptotic behavior]{Asymptotic behavior and representation of solutions to a Volterra kind of equation with a singular kernel}

\author{Rodrigo Ponce}
\address{Universidad de Talca\\Instituto de Matem\'atica y F\'isica\\Casilla 747, Talca-Chile.}
\email{rponce@inst-mat.utalca.cl}

\thanks{The research of the first author is partially supported by Fondecyt Iniciaci\'on \#11130619.}

\author{Mahamadi Warma}
\address{University of Puerto Rico  (Rio Piedras Campus)\\College of Natural Sciences\\
Department of Mathematics\\PO Box 70377 San Juan PR 00936-8377 (USA)}
\email{mahamadi.warma1@upr.edu, mjwarma@gmail.com}

\thanks{The work of the second author is partially supported by the Air Force Office of Scientific Research under the Award No: FA9550-15-1-0027.}



\begin{document}


\begin{abstract}
Let $A$ be a densely defined closed, linear $\omega$-sectorial operator of angle $\theta\in [0,\frac{\pi}{2})$ on a Banach space $X$ for some $\omega\in\R$. We give an explicit representation (in terms of some special functions) and study the precise asymptotic behavior as time goes to infinity of solutions to the following diffusion equation with memory:
$\displaystyle u'(t)=Au(t)+(\kappa\ast Au)(t), \, t >0$, $u(0)=u_0$, associated with the (possible) singular kernel
 $\kappa(t)=\alpha e^{-\beta t}\frac{t^{\mu-1}}{\Gamma(\mu)},\;\;t>0$, where $\alpha\in\R$, $\alpha\ne 0$, $\beta\ge  0$ and $0<\mu\le 1$.
\end{abstract}
\maketitle

\section{Introduction}

In the present paper, we consider the following Volterra kind of system
\begin{equation}\label{eq-IV-0}
\begin{cases}
\displaystyle u'(t)=Au(t)+\int_0^t \kappa(t-s)Au(s)ds, \, t>0,\\
u(0)=u_0,
\end{cases}
\end{equation}
where $A$ with domain $D(A)$ is a densely defined linear sectorial operator on a Banach space $X$, $u_0$ is a given function in $X$ and the possible (singular) kernel $\kappa$ is given by
\begin{align}\label{kernel}
\kappa(t)=\alpha e^{-\beta t}\frac{t^{\mu-1}}{\Gamma(\mu)},\;\;t>0,
\end{align}
with $\alpha\in\mathbb{R}$, $\alpha\ne 0$, $\beta\ge 0$ and $0<\mu\le 1$. We mention that the kernel $\kappa\in L_{\loc}^1([0,\infty))$, and is even in $L^1((0,\infty))$ if $\beta>0$, but it is not in $W_{\loc}^{1,1}([0,\infty))$ if $0<\mu<1$.
 It is straightforward to verify that the system \eqref{eq-IV-0} is equivalent to the integral equation
\begin{equation}\label{eqv-sys}
\displaystyle u(t)=u_0+\int_0^t \left(1+1\ast \kappa\right)(t-s)Au(s)ds, \, \;\;t\ge 0.
\end{equation}
By \cite[Chapter I]{Pr}, the well-posedness of the system \eqref{eq-IV-0} (or equivalently, the integral equation \eqref{eqv-sys}) is equivalent to the existence of a family of bounded linear operators $(S_{\alpha,\beta}^\mu(t))_{t\ge 0}$ on $X$, which we shall call {\bf $(\alpha,\beta,\mu)$-resolvent family}, verifying the following properties.
\begin{itemize}
\item $S_{\alpha,\beta}^\mu(0)=I$ and $S_{\alpha,\beta}^\mu(\cdot)$ is strongly continuous on $[0,\infty)$.

\item $S_{\alpha,\beta}^\mu(t)x\in D(A)$ and $S_{\alpha,\beta}^\mu(t)Ax=AS_{\alpha,\beta}^\mu(t)x$ for all $x\in D(A)$ and $t\ge 0$.

\item For all $x\in D(A)$ and $t\ge 0$, the resolvent equation holds:
\begin{align}\label{reo-eq}
S_{\alpha,\beta}^\mu(t)x=x+\int_0^t \left(1+1\ast k\right)(t-s)AS_{\alpha,\beta}^\mu(s)x\;ds.
\end{align}
\end{itemize}

In that case, for every $u_0\in X$, the unique (mild) solution of \eqref{eq-IV-0} is given by
\begin{equation}\label{uni-sol}
u(t)=S_{\alpha,\beta}^\mu(t)u_0 ,\;\;t\ge 0.
\end{equation}
We give the idea how to get \eqref{uni-sol} without further detail.
In fact,  uniqueness of solutions is easy to establish. Taking the Laplace transform of both sides of the first equation in \eqref{eq-IV-0} we get that
\begin{align*}
\lambda \widehat u(\lambda)-u(0)=A\widehat u(\lambda)+\frac{\alpha}{(\lambda+\beta)^{\mu}}A\widehat u(\lambda).
\end{align*}
Thus
\begin{align}\label{Bb1}
&\frac{(\lambda+\beta)^{\mu}+\alpha}{(\lambda+\beta)^\mu}\left(\frac{\lambda(\lambda+\beta)^{\mu}}{(\lambda+\beta)^\mu+\alpha}-A\right)\widehat u(\lambda)\notag\\
=&[g(\lambda)]^{-1}\Big(h(\lambda)-A\Big)\widehat u(\lambda)=u_0,
\end{align}
where we have set
\begin{align*}
g(\lambda):=\frac{(\lambda+\beta)^\mu}{(\lambda+\beta)^\mu+\alpha}\;\mbox{ and }\; h(\lambda):=\frac{\lambda(\lambda+\beta)^\mu}{(\lambda+\beta)^\mu+\alpha}=\lambda g(\lambda).
\end{align*}
Note that $g(\lambda)$ and $h(\lambda)$ are well-defined for all $\lambda$ with $\op(\lambda)>0$.
Taking also the Laplace transform of both sides of \eqref{reo-eq} with $x=u_0$ we get that
\begin{align*}
\widehat S_{\alpha,\beta}^\mu(\lambda)u_0=\frac{1}{\lambda}u_0+\left(\frac{1}{\lambda}+\frac{\widehat k(\lambda)}{\lambda} \right)A\widehat S_{\alpha,\beta}^\mu(\lambda)u_0.
\end{align*}
Calculating, we obtain that
\begin{align}\label{Bb2}
&\frac{(\lambda+\beta)^{\mu}+\alpha}{(\lambda+\beta)^\mu}\left(\frac{\lambda(\lambda+\beta)^{\mu}}{(\lambda+\beta)^\mu+\alpha}-A\right)\widehat S_{\alpha,\beta}^\mu(\lambda)u_0\notag\\
=&[g(\lambda)]^{-1}\Big(h(\lambda)-A\Big)\widehat S_{\alpha,\beta}^\mu(\lambda)u_0=u_0.
\end{align}
It follows from \eqref{Bb1} and \eqref{Bb2} that if $h(\lambda)\in\rho(A)$, then
\begin{align}\label{Bb3}
\widehat u(\lambda)=\widehat S_{\alpha,\beta}^\mu(\lambda)u_0=g(\lambda)\Big(h(\lambda)-A\Big)^{-1}u_0.
\end{align}
Now taking the inverse Laplace transform of \eqref{Bb3} we get that the unique (mild) solution of \eqref{eq-IV-0} is given by \eqref{uni-sol}. For more details on this topic we refer to the monograph \cite{Pr}.


Next, we assume that the system \eqref{eq-IV-0} is well-posed and we denote by $S_{\alpha,\beta}^\mu$ the associated $(\alpha,\beta,\mu)$-resolvent family.

\begin{definition}
We shall say that $S_{\alpha,\beta}^\mu$ is {\bf exponentially bounded, or of type $(M,\widetilde\omega)$}, if there exist two constants $M>0$ and $\widetilde\omega\in\R$ such that
\begin{align}\label{exp-bd}
\|S_{\alpha,\beta}^\mu(t)\|\le M e^{\widetilde\omega t},\;\;\;\forall\;t\ge 0.
\end{align}
The resolvent family $S_{\alpha,\beta}^\mu$ will be said to be {\bf uniformly exponentially stable} if  \eqref{exp-bd} holds with some constants $M>0$ and $\widetilde\omega<0$.
\end{definition}

As we have mentioned above, the well-posedness of the system \eqref{eq-IV-0} and many other properties of  resolvent families have been intensively studied in the monograph \cite{Pr}. Similar problems have been also considered in \cite{Ch-Li-Xi-11,Cor,Jac,Li-00,Li-Po-11} and their references. The asymptotic behavior as time goes to $\infty$ of solutions to some similar integro-differential equations in finite dimension involving kernels as the one in \eqref{kernel} have been also studied in \cite{VeZa}.

The main concerns in the present paper are the following.
\begin{enumerate}
\item[(1)] Study the precise asymptotic behavior as time goes to infinity of  solutions to the system \eqref{eq-IV-0}. That is, we would like to characterize the minimum real number $\widetilde\omega$ in the estimate \eqref{exp-bd}. Of particular interest will be to find  some precise conditions on the parameters $\alpha$, $\beta$, $\mu$ and the operator $A$ that shall imply the uniform exponential stability of the family $S_{\alpha,\beta}^\mu$, and hence, of solutions to the system \eqref{eq-IV-0}.

\item[(2)] Find an explicit representation of solutions to the system \eqref{eq-IV-0} in terms of some special functions. The Mittag-Leffler functions and its generalizations will be the natural candidate.

\end{enumerate}

We mention that if $\mu=1$ and $\beta>0$ in \eqref{kernel}, that is, $\kappa(t)=\alpha e^{-\beta t}$, then the asymptotic behavior of the associated resolvent family
has been completely studied  in \cite{Ch-Xi-Li-09} by using some semigroups method. See also \cite[Chapter VI]{En-Na-00} where the well-posedness of the system \eqref{eq-IV-0} for a general kernel $\kappa\in W^{1,1}((0,\infty))$ has been obtained by using again the method of semigroups.
One of our concerns is to extend the results contained in \cite{Ch-Xi-Li-09} to the case $0<\mu<1$, where the method of semigroups  cannot be used. Note that in our situation $\kappa\not\in W_{\loc}^{1,1}([0,\infty))$.
If $\beta=0$ and $0<\mu<1$, then it is well-known (see e.g. \cite{Cu-07,Go-Ma00,Ha-Ma-Sa-11,Ke-Li-Wa-13,Mi-Ro-93} and their references) that the corresponding $(\alpha,0,\mu)$-resolvent family is never uniformly exponentially stable. In this case, solutions may decay but only at most polynomial. We shall see here that the situation is different if $\beta>0$.  More precisely, assuming that the operator $A$ is $\omega$-sectorial of angle $\theta$ (see Section \ref{main-resl} below for the definition) for some $\omega<0$, $0\le \theta\le\widetilde\mu\frac{\pi}{2}$, where $0<\mu<\widetilde\mu\le 1$, and that $\beta+\omega\le 0$, we obtain the following result.
\begin{enumerate}
\item  If $\alpha>0$, then the family is uniformly exponentially stable with $\widetilde\omega=-\beta$.

\item If $\alpha<0$ and $\alpha+\beta^\mu\ge |\alpha|$, then the family is exponentially bounded with exponential bound $\widetilde\omega=-\left(\beta-(\alpha\omega)^{\frac{1}{\mu+1}}\right)$. This will imply the uniform exponential stability of the family if in addition $\beta^{\mu+1}>\alpha\omega$.
\end{enumerate}

The rest of the paper is structured as follows. In Section \ref{main-resl} we state the main results where some generation  results of $(\alpha,\beta,\mu)$-resolvent families and the precise asymptotic behavior as time goes to infinity of the resolvent family have been obtained. Section \ref{sec-pr-gen} contains the proof of the generation theorem. In Section \ref{sec-pr-2th} we give the proof of the asymptotic behavior as time goes to infinity of the resolvent family. Finally in Section \ref{sec-par-ca}, assuming that $A=\rho I$ for some $\rho\in\R$ or $A$ is a self-adjoint operator on $L^2(\Omega)$ (where $\Omega\subset\R^N$ is a bounded open set) with compact resolvent, we obtain an explicit representation of solutions to the system \eqref{eq-IV-0} (and hence of the associated resolvent family) in terms of Mittag-Leffler type of functions.

\section{Main results}\label{main-resl}

Let $X$ be a Banach space with norm $\|\cdot\|$.  For a closed, linear and densely defined operator $A$ on $X$, we denote by $\sigma(A)$ and $\rho(A)$ the spectrum and the resolvent set of $A$, respectively. The operator $A$ is said to be {\bf $\omega$-sectorial of angle $\theta$} if there exist
$\theta\in [0,\pi/2)$ and $\omega\in \mathbb{R}$ such that
\begin{align*}
w+\Sigma_{\theta}:=\left\{\omega+\lambda,\; \lambda\in \mathbb{C}:|\arg(\lambda)|<\theta+\frac{\pi}{2}\right\}\subset\rho(A),
\end{align*}
and one has the estimate
\begin{align*}
\|(\lambda-A)^{-1}\|\leq \frac{M}{|\lambda-\omega|}, \quad \forall\;\lambda \in \{\omega+\Sigma_{\theta}\}\setminus\{\omega\}.
\end{align*}

If $\omega=0$, we shall only say that $A$ is {\bf sectorial of angle $\theta$}. More details on sectorial operators can be found in \cite{En-Na-00,Haa-06} and they references.

In this section we state the main results of the paper.  We start with the generation theorem.

\begin{theorem}\label{generation}
Let $\alpha\in\mathbb{R}$, $\alpha\ne 0$, $\beta\ge 0$ and $0<\mu\leq 1$. Assume any one of the following two conditions.

\begin{enumerate}
\item  $\alpha> 0$ and $A$ is a sectorial operator of angle $\mu\frac{\pi}{2}$.

\item $\alpha<0,$ $\alpha+\beta^\mu\ge |\alpha|$ and $A$ is a sectorial operator of angle $\mu\frac{\pi}{2}$.
\end{enumerate}
Then $A$ generates an $(\alpha,\beta,\mu)$-resolvent family $(S_{\alpha,\beta}^\mu(t))_{t\geq 0}$ of type $(M,\widetilde\omega)$ for every $\widetilde\omega>0$.
\end{theorem}

\begin{corollary}\label{cor-generation}
Let $\alpha\in\mathbb{R}$, $\alpha\ne 0$, $\beta\ge 0$, $0<\mu\leq 1$ and $\omega\in\R$. Assume any one of the following two conditions.

\begin{enumerate}
\item  $\alpha> 0$ and $A$ is an $\omega$-sectorial operator of angle $\mu\frac{\pi}{2}$.

\item $\alpha<0,$ $\alpha+\beta^\mu\ge |\alpha|$ and $A$ is an $\omega$-sectorial operator of angle $\mu\frac{\pi}{2}$.
\end{enumerate}
Then $A$ generates an $(\alpha,\beta,\mu)$-resolvent family $(S_{\alpha,\beta}^\mu(t))_{t\geq 0}$ of type $(M,\widetilde \omega)$ for some $\widetilde\omega>0$.
\end{corollary}

As we have mentioned above, the asymptotic behavior as time goes to infinity of the resolvent family for the case $\mu=1$ has been completely studied in \cite{Ch-Xi-Li-09}. Therefore we concentrate on the case $0<\mu<1$.
The following estimates of the resolvent family is the second main result of the paper.

\begin{theorem}\label{th2-P}
Let $\alpha\in\mathbb{R}$, $\alpha\ne 0$, $\beta\ge 0$,  $0<\mu<\widetilde\mu \leq 1$ and $\omega<0.$ Assume that $\beta+\omega\leq 0.$ Then the following assertions hold.
\begin{enumerate}
\item If $\alpha>0$ and $A$ is $\omega$-sectorial of angle $\widetilde\mu\frac{\pi}{2},$ then
      there exists a constant $C>0$  such that the  resolvent family $(S_{\alpha,\beta}^\mu(t))_{t\geq 0}$ generated by $A$ satisfies the estimate
      \begin{align}\label{eq-asymp1-P}
         \|S_{\alpha,\beta}^\mu(t)\|\leq C e^{-\beta t}, \quad \forall\;t\geq 0.
      \end{align}
\item If $\alpha<0$, $\alpha+\beta^\mu\ge |\alpha|$  and $A$ is $\omega$-sectorial  of angle $\widetilde\mu\frac{ \pi}{2}$, then
      there exists a constant $C>0$  such that the resolvent family $(S_{\alpha,\beta}^\mu(t))_{t\geq 0}$ generated by $A$ satisfies the estimate
      \begin{align}\label{eq-asymp2-P}
         \|S_{\alpha,\beta}^\mu(t)\|\leq C\left(1+\alpha\omega t^{\mu+1}\right)e^{-(\beta-(\alpha\omega)^{\frac{1}{\mu+1}}) t}, \quad\forall\; t\geq 0.
      \end{align}
\end{enumerate}
\end{theorem}

\section{Proof of Theorem \ref{generation} and Corollary \ref{cor-generation}}\label{sec-pr-gen}

Before given the proof of the generation theorem we need some intermediate results. First, we recall the following fundamental result  from \cite[Chapter I, Theorem 1.3]{Pr}.

\begin{theorem}\label{theo-Pruss}
Let $\gota\in L_{\operatorname{loc}}^1([0,\infty))$ satisfy
\begin{equation}\label{cond-a}
\int_0^\infty e^{-\widetilde\omega t}|\gota (t)|\;dt<\infty.
\end{equation}
Let $A$ with domain $D(A)$ be a linear operator on $X$ with dense domain. Then $A$ generates a resolvent family {\em (associated with $\gota$)} of type $(M,\widetilde\omega)$ on $X$ if and only if the following two conditions hold.
\begin{enumerate}
\item $\widehat \gota(\lambda)\ne 0$ and $\frac{1}{\widehat \gota(\lambda)}\in \rho(A)$ for all $\lambda>\widetilde\omega$.

\item The mapping $\lambda\mapsto H(\lambda):=\frac{1}{\lambda}\left(I-\widehat \gota(\lambda))A\right)^{-1}$ satisfies the estimate
\begin{align*}
\|H^{(n)}(\lambda)\|\le \frac{Mn!}{(\lambda-\widetilde\omega)^{n+1}},\;\;\lambda>\widetilde\omega,\; n\in\N_0:=\N\cup\{0\},
\end{align*}
where $H^{(n)}(\lambda)=\frac{d^nH}{d\lambda^n}(\lambda)$.
\end{enumerate}
\end{theorem}

The following result will be also useful.

\begin{lemma}\label{lem1}
Let $\alpha\in\mathbb{R}$, $\beta\ge 0$,  $0<\mu\leq 1$ and define $g(\lambda):=\frac{(\lambda+\beta)^\mu}{(\lambda+\beta)^\mu+\alpha}$ for all $\lambda$ with $\op(\lambda)>0$. Then the following assertions hold.
\begin{enumerate}
  \item If $\alpha\geq 0,$ then $|g(\lambda)|\leq 1$ .
  \item If $\alpha<0$ and $\alpha+\beta^\mu\ge |\alpha|,$ then $|g(\lambda)|\leq \frac{\beta^\mu}{\alpha+\beta^\mu}\le  \frac{\beta^\mu}{|\alpha|}$.
\end{enumerate}
\end{lemma}

\begin{proof}
Let $\alpha\in\mathbb{R}$, $\beta\ge 0$ and $0<\mu\leq 1$.

(a) This assertion is obvious.

(b) Now assume that $\alpha<0$ and $\beta^\mu+\alpha\ge |\alpha|$.  First, observe that
\begin{align}\label{A1}
|g(\lambda)|=\left|1-\frac{\alpha}{(\lambda+\beta)^\mu+\alpha}\right|
\leq1+\frac{|\alpha|}{|(\lambda+\beta)^\mu-\beta^\mu+(\alpha+\beta^\mu)|}.
\end{align}
We claim that
\begin{align}\label{claim}
{\rm Re}[(\lambda+\beta)^\mu-\beta^\mu]>0\;\mbox{ for all }\; \op(\lambda)>0.
\end{align}
We show first that
\begin{align}\label{In-Eq}
\op[(\lambda+\beta)^\mu]>({\rm Re}(\lambda)+\beta)^\mu\;\mbox{ for all } \op(\lambda)>0.
\end{align}
 Indeed, let $z:=\lambda+\beta.$ We have to show that ${\rm Re}(z^\mu)>({\rm Re}(z))^\mu,$ that is, $|z|^\mu\cos(\theta\mu)>(|z|\cos(\theta))^\mu,$ where $z=|z|\cos(\theta)$ with $|\theta|<\frac{\pi}{2}$. Let $f(\mu):=\cos(\mu\theta)-\cos^\mu(\theta),$ $0\le \mu\leq 1.$ Then
\begin{align*}
f''(\mu)=-\theta^2\cos(\mu\theta)-(\ln(\cos(\theta)))^2\cos^\mu(\theta)<0\;\mbox{ for all }\; 0<\mu\leq 1,
\end{align*}
and this implies that $f$ is a concave function. Since $f(0)=f(1)=0$, we have that the graph of $f$ is above the straight line joining the points $(0,f(0))=(0,0)$ and $(1,f(1))=(1,0).$ This implies that $f(\mu)\geq 0$ for all $0\leq\mu\leq 1$. Thus $\cos(\mu\theta)\geq\cos^\mu(\theta)$ for all $0\le \mu\leq 1$ and we have shown \eqref{In-Eq}. The claim \eqref{claim} follows from \eqref{In-Eq}.

On the other hand, since by assumption $\alpha+\beta^\mu\ge|\alpha|>0$, it follows from \eqref{claim} that
\begin{align}\label{eqA.1}
|{\rm Re}[(\lambda+\beta)^\mu-\beta^\mu+(\alpha+\beta^\mu)]|>\alpha+\beta^\mu.
\end{align}
Using \eqref{eqA.1}, we get from \eqref{A1} that
\begin{align*}
|g(\lambda)|\leq 1+\frac{|\alpha|}{\alpha+\beta^\mu}=\frac{\beta^\mu}{\alpha+\beta^\mu}\le \frac{\beta^\mu}{|\alpha|},
\end{align*}
and the proof is finished.
\end{proof}

We also need the following result.

\begin{lemma}\label{lem-h}
Let $\alpha\in\mathbb{R}$, $\alpha\ne 0$, $\beta\ge 0$,  $0<\mu\leq 1$ and define the function
\begin{align}\label{func-h}
h(\lambda)=\frac{\lambda(\lambda+\beta)^\mu}{(\lambda+\beta)^\mu+\alpha},\;\lambda\in\C,\;\op(\lambda)>0.
\end{align}
Assume any one of the following two conditions.
\begin{enumerate}
  \item $\alpha>0$.
  \item $\alpha<0$ and $\alpha+\beta^\mu\ge |\alpha|$.
\end{enumerate}
Then
\begin{align}\label{arg}
|\arg(h(\lambda))|\le (1+\mu)|\arg(\lambda)|.
\end{align}
\end{lemma}

\begin{proof}
Let $h$ be given  by \eqref{func-h} where $\lambda=re^{i\theta}$ with $|\theta|<\frac{\pi}{2}$ and $r>0$. Without any restriction we may assume that $\theta\ge 0$.
Then
\begin{align}\label{cal-arg}
\arg(h(re^{i\theta}))=&{\rm Im}({\ln}(h(re^{i\theta})))={\rm Im}\int_0^\theta\frac{d}{dt}{\ln}(h(re^{it}))dt\notag\\
=&{\rm Im}\int_0^\theta\frac{h'(re^{it})ire^{it}}{h(re^{it})}dt.
\end{align}
Moreover a simple calculation gives
\begin{align*}
\lambda\frac{h'(\lambda)}{h(\lambda)}=1+\left(\frac{\alpha\mu}{(\lambda+\beta)^\mu+\alpha}\right)\left(\frac{\lambda}{\lambda+\beta}\right).
\end{align*}
(a) Assume that $\alpha> 0$. Then
\begin{align}\label{cal-arg-2}
\left|\lambda\frac{h'(\lambda)}{h(\lambda)}\right|\leq 1+\mu \left|\frac{\alpha}{(\lambda+\beta)^\mu+\alpha}\right| \left|\frac{\lambda}{\lambda+\beta}\right|\leq 1+\mu.
\end{align}
Using \eqref{cal-arg-2} we get from \eqref{cal-arg} that
\begin{align*}
\left|{\rm Im}\int_0^\theta\frac{h'(re^{it})ire^{it}}{h(re^{it})}dt\right|\leq\int_0^\theta\left|\frac{h'(re^{it})ire^{it}}{h(re^{it})}\right|dt
\leq \int_0^\theta (1+\mu)dt\leq (1+\mu)\theta,
\end{align*}
and we have shown \eqref{arg}.

(b) Now assume that $\alpha<0$ and $\beta^\mu+\alpha\ge |\alpha|$. Then using \eqref{eqA.1} we get that
\begin{align}\label{cal-arg-3}
\left|\lambda\frac{h'(\lambda)}{h(\lambda)}\right|\leq 1+\mu \left|\frac{\alpha}{(\lambda+\beta)^\mu+\alpha}\right| \left|\frac{\lambda}{\lambda+\beta}\right|\leq 1+\frac{\mu|\alpha|}{\alpha+\beta^\mu}\le 1+\mu.
\end{align}
It follows from \eqref{cal-arg} and \eqref{cal-arg-3} that
\begin{align*}
\left|{\rm Im}\int_0^\theta\frac{h'(re^{it})ire^{it}}{h(re^{it})}dt\right|&\leq\int_0^\theta\left|\frac{h'(re^{it})ire^{it}}{h(re^{it})}\right|dt\\
&\leq \int_0^\theta 1+\frac{\mu|\alpha|}{\alpha+\beta^\mu}dt\leq
(1+\mu)\theta.
\end{align*}
We have shown \eqref{arg} and the proof is finished.
\end{proof}

\begin{proof}[\bf Proof of Theorem \ref{generation}]
First we note that compare with Theorem \ref{theo-Pruss} we have that
\begin{align*}
\gota(t)=(1+1\ast \kappa)(t),\;\;t>0,
\end{align*}
so that its Laplace transform is given by
\begin{align*}
\widehat\gota(\lambda)=\frac{1}{\lambda}+\frac{\widehat \kappa(\lambda)}{\lambda}=\frac{(\lambda+\beta)^\mu+\alpha}{\lambda(\lambda+\beta)^\mu}.
\end{align*}
It is clear that for every $\widetilde\omega>0$ we have that
\begin{align*}
\int_0^\infty e^{-\widetilde\omega t}|\gota(t)|\;dt=\int_0^\infty e^{-\widetilde\omega t}|(1+1\ast \kappa)(t)|\;dt<\infty.
\end{align*}
Hence, we have to show that the two conditions in Theorem \ref{theo-Pruss} are satisfied.  It is easy to see that under the assumptions (a) or (b) we have that $\widehat\gota(\lambda)\ne 0$ for all $\lambda$ with $\mbox{Re}(\lambda)>0$.

Next, we claim that
\begin{equation}\label{eq-res}
\frac{1}{\widehat\gota(\lambda)}=h(\lambda)=\frac{ \lambda(\lambda+\beta)^\mu}{(\lambda+\beta)^\mu+\alpha}\in\rho(A)\;\;\mbox{ for all }\;\lambda\;\mbox{ with }\;\op(\lambda)>0.
\end{equation}
It follows from Lemma \ref{lem-h} that in both cases (a) and (b), we have that
\begin{align*}
h(\lambda)\in \Sigma_{\frac{\pi\mu}{2}}\;\mbox{ for all }\;\lambda\;\mbox{ with }\;\op(\lambda)>0.
\end{align*}
This implies that
the function $H(\lambda)=g(\lambda)(h(\lambda)-A)^{-1}$ is well defined, where  $g(\lambda)$ is given by
\begin{align*}
g(\lambda)=\frac{(\lambda+\beta)^\mu}{(\lambda+\beta)^\mu+\alpha}=\frac{h(\lambda)}{\lambda }.
\end{align*}
Since $A$ is a sectorial operator of angle $\mu \frac{\pi}{2}$ (in both cases (a) and (b)),  we have that there exists a constant $M>0$ such that for all $\lambda$ with $\op(\lambda)>0,$
\begin{align}\label{A2}
\|\lambda H(\lambda)\|= |h(\lambda)|\left\|(h(\lambda)-A)^{-1}\right\|\leq M\frac{|h(\lambda)|}{|h(\lambda)|}= M.
\end{align}
Moreover a simple calculation gives that in both cases,
\begin{align*}
\lambda^2H'(\lambda)=&\mu\tfrac{\lambda}{\lambda+\beta}\lambda H(\lambda)-\mu\tfrac{\lambda}{\lambda+\beta}g(\lambda)\lambda H(\lambda)-\lambda^2 H(\lambda)^2\\
&-\mu\tfrac{\lambda}{\lambda+\beta}\lambda^2 H(\lambda)^2+g(\lambda)\lambda^2 H(\lambda)^2\mu \tfrac{\lambda}{\lambda+\beta}.
\end{align*}
Note that in the case (b) we have that $\alpha+\beta^\mu\ge (1+\mu)\alpha+\beta^\mu\ge |\alpha|+\mu\alpha>0$.
Since the function $g$ is bounded (by Lemma \ref{lem1}), then using \eqref{A2}, we get that there exists a constant $M_1>0$ such that
\begin{align}\label{A3}
\|\lambda^2 H'(\lambda)\|\leq M_1\;\;\mbox{ for all }\;\lambda\;\mbox{ with }\;\op(\lambda)>0.
\end{align}
Combining \eqref{A2} and \eqref{A3} we get that there exists a constant $M>0$ such that
\begin{align}\label{A4}
\|\lambda H(\lambda)+\|\lambda^2 H'(\lambda)\|\leq M\;\;\mbox{ for all }\;\lambda\;\mbox{ with }\;\op(\lambda)>0.
\end{align}
By \cite[Proposition 0.1]{Pr}, the estimate \eqref{A4} implies that
\begin{align}\label{A5}
\|H^{(n)}(\lambda)\|\le \frac{Mn!}{\lambda^{n+1}},\;\;\;\;\forall\;\lambda>0,\;n\in\N_0.
\end{align}
From \eqref{A5} we also have that for every $\widetilde\omega>0$,
\begin{align*}
\|H^{(n)}(\lambda)\|\le \frac{Mn!}{(\lambda-\widetilde\omega)^{n+1}},\;\;\;\;\forall\;\lambda>\widetilde\omega,\;n\in\N_0.
\end{align*}
Finally, it follows from Theorem \ref{theo-Pruss} that $A$ generates an $(\alpha,\beta,\mu)$-resolvent family $(S_{\alpha,\beta}^\mu(t))_{t\geq 0}$ of type $(M,\widetilde\omega)$ and the proof is finished.
\end{proof}

\begin{proof}[\bf Proof of Corollary \ref{cor-generation}]
Assume that $\alpha>0$, or $\alpha<0$ and $\alpha+\beta^\mu\ge |\alpha|$ and that $A$ is $\omega$-sectorial of angle $\mu\frac{\pi}{2}$. The claim follows from the decomposition $A=(\omega I+A)-\omega I$, using Theorem \ref{generation} and the perturbation result of resolvent families contained in \cite[Corollary 3.2]{LiSa}.
\end{proof}

We make some comments about the results obtained in Theorem \ref{generation} and Corollary \ref{cor-generation}.

\begin{remark}
We notice that even if we assume that the operator $A$ is $\omega$-sectorial for some $\omega<0$, then uniform exponential stability of the family $(S_{\alpha,\beta}^\mu(t))_{t\geq 0}$ cannot be derived from Theorem \ref{generation} and Corollary \ref{cor-generation}. In fact, the integral condition \eqref{cond-a} on the kernel $\gota(t)=(1+1\ast \kappa)(t)$ is only satisfied for $\widetilde\omega>0$. Therefore some new ideas are needed in the study of the uniform exponential stability of the family which is one of the main goals of the present paper.
\end{remark}

\section{Proof of Theorems \ref{th2-P}}\label{sec-pr-2th}

In this section we give the proof of Theorems \ref{th2-P}. For this we need the following result.

\begin{lemma}\label{lem-ex}
Let $\beta\ge 0$, $\alpha\in\R$, $\alpha\ne 0$, $0<\mu\le 1$ and define the function
\begin{align}\label{h-til}
\widetilde h(\lambda):=\frac{(\lambda-\beta)\lambda^\mu}{\lambda^\mu+\alpha}.
\end{align}
If $\op(\lambda)<0$ and $|\lambda^{\mu}|\ge 2|\alpha|$, then
\begin{align}\label{argu}
|\arg(\widetilde  h(\lambda))|\le (1+\mu)|\arg(\lambda)|.
\end{align}
\end{lemma}

\begin{proof}
We proceed as in the proof of Lemma \ref{lem-h}.  Let $\widetilde h$ be given by \eqref{h-til} with $\lambda=re^{i\theta}$ where $|\theta|> \frac{\pi}{2}$ and $r^\mu\ge 2|\alpha|$. Without any restriction we may assume that $\theta>\frac{\pi}{2}$. A simple calculation gives that
\begin{align}\label{cal-1}
\lambda\frac{\widetilde h'(\lambda)}{\widetilde h(\lambda)}=\frac{\lambda}{\lambda-\beta}+\mu\frac{\alpha}{\lambda^\mu+\alpha}.
\end{align}
Recall that
\begin{align*}
{\rm arg}(\widetilde h(re^{i\theta}))={\rm Im}({\ln}(\widetilde h(re^{i\theta})))={\rm Im}\int_0^\theta\frac{\widetilde h'(re^{it})ire^{it}}{\widetilde h(re^{it})}dt.
\end{align*}
Since $\op(\lambda)<0$ and $|\lambda^\mu|\ge 2|\alpha|$, we have that $|\lambda^\mu+\alpha|\ge \left||\lambda^\mu|-|\alpha|\right|\ge |\alpha|$. Thus it follows from \eqref{cal-1} that
\begin{align*}
\left|\lambda\frac{\widetilde h'(\lambda)}{\widetilde h(\lambda)}\right|\leq 1+\mu.
\end{align*}
Therefore
\begin{align*}
\left|{\rm Im}\int_0^\theta\frac{h'(re^{it})ire^{it}}{h(re^{it})}dt\right|\leq\int_0^\theta\left|\frac{h'(re^{it})ire^{it}}{h(re^{it})}\right|dt
\leq \int_0^\theta (1+\mu)dt\leq (1+\mu)\theta,
\end{align*}
and the proof is finished.
\end{proof}

\begin{proof}[\bf Proof of Theorem \ref{th2-P}]
Let $\beta\ge 0$,  $\alpha\in\R$, $\alpha\ne 0$, $0<\mu<\widetilde\mu\leq 1$ and $\omega<0$.

(a) Assume that $\alpha>0$. Since $A$ is $\omega$-sectorial of angle $0<\theta\le \widetilde\mu \frac{\pi}{2}$, it follows from Theorem \ref{generation} and Corollary \ref{cor-generation} that $A$ generates an $(\alpha,\beta,\mu)$-resolvent family $(S_{\alpha,\beta}^\mu(t))_{t\geq 0}$. Recall the function $h$ given by
\begin{align*}
h(\lambda):=\frac{\lambda(\lambda+\beta)^\mu}{(\lambda+\beta)^\mu+\alpha}.
\end{align*}
We have that for every $\lambda\in\C$ such that $h(\lambda)\in \omega+\Sigma_\theta$,
the resolvent $\left(\frac{\lambda(\lambda+\beta)^\mu}{(\lambda+\beta)^\mu+\alpha}I-A\right)^{-1}$ exists
and the Laplace transform of $S_{\alpha,\beta}^\mu(t)$ is given by
\begin{align}\label{J1}
\widehat S_{\alpha,\beta}^\mu(\lambda)=\frac{(\lambda+\beta)^\mu}{(\lambda+\beta)^\mu+\alpha}\left(\frac{\lambda(\lambda+\beta)^\mu}{(\lambda+\beta)^\mu+\alpha}I-A\right)^{-1}.
\end{align}
By the uniqueness of the Laplace transform, it follows from \eqref{J1} that
\begin{align}\label{uniq-lp}
S_{\alpha,\beta}^\mu(t)=e^{-\beta t}G_{\alpha,\beta}^\mu(t)\;\mbox{ for all }\; t\ge 0,
\end{align}
where the Laplace transform of $G_{\alpha,\beta}^\mu(t)$ is given, for all $\lambda\in\C$ such that $h(\lambda-\beta) \in \omega+\Sigma_\theta$, by
\begin{align*}
\widehat G_{\alpha,\beta}^\mu(\lambda)=\frac{\lambda^\mu}{\lambda^\mu+\alpha}\left(\frac{(\lambda-\beta)\lambda^\mu}{\lambda^\mu+\alpha}I-A\right)^{-1}.
\end{align*}
Let
\begin{align*}
h(\lambda-\beta)=\widetilde h(\lambda)=\frac{\lambda^\mu(\lambda-\beta)}{\lambda^\mu+\alpha}.
\end{align*}
There exists a constant $M>0$ such that for all $\lambda\in\C$ with $\widetilde h(\lambda)\in  \omega+\Sigma_\theta$, we have the estimate
\begin{align}\label{eq1.1-P}
\left\|\left(\frac{(\lambda-\beta)\lambda^\mu}{\lambda^\mu+\alpha}I-A\right)^{-1}\right\|\leq& \frac{M}{\left|\frac{(\lambda-\beta)\lambda^\mu}{\lambda^\mu+\alpha}-\omega\right|}\notag\\
=&\frac{M|\lambda^\mu+\alpha|}{|\lambda^{\mu+1}-(\beta+\omega)\lambda^\mu-\alpha\omega|}.
\end{align}
Note that $\alpha\omega<0$ and by hypothesis $-(\omega+\beta)\ge 0$.
Let
\begin{align*}
g(t)=t^{\mu+1}, \;\;t>0.
\end{align*}
We exploit some ideas from the proof of \cite[Theorem 1]{Cu-07}. Recall that $A$ is $\omega$-sectorial of angle $\theta=\widetilde\mu\frac{\pi}{2}$.
Using the inversion formula for the Laplace transform, we have that
\begin{align}\label{eq1.2-P}
G_{\alpha,\beta}^\mu(t)=\frac{1}{2\pi i}\int_\gamma e^{\lambda t} \frac{\lambda^\mu}{\lambda^\mu+\alpha}\left(\frac{(\lambda-\beta)\lambda^\mu}{\lambda^\mu+\alpha}I-A\right)^{-1}d\lambda,
\end{align}
where $\gamma$ is a positively oriented path whose support $\Gamma$ is given by
\begin{align*}
\Gamma:=\{\lambda:\;\lambda\in\C,\;\;\;\lambda^{\mu+1}\;\mbox{ belongs to the boundary of }\; B_{\frac{1}{g(t)}},\; t>0\},
\end{align*}
 where for $\delta>0$,
\begin{align*}
B_\delta:=\{\delta+\Sigma_\theta\}\cup S_\phi,
\end{align*}
and
\begin{align*}
S_\phi:=\{\lambda\in \C:\;|\mbox{arg}(\lambda)|<\phi\},\;\;(\mu+1)\frac{\pi}{2}<\phi<\frac{\pi}{2}+\theta.
\end{align*}
 It follows from Lemma \ref{lem-ex} that for such $\lambda$, we have that $\widetilde h(\lambda)\in \omega+\Sigma_{\theta}$. Hence, the resolvent $\left(\frac{(\lambda-\beta)\lambda^\mu}{\lambda^\mu+\alpha}I-A\right)^{-1}$ is well-defined. Thus the representation \eqref{eq1.2-P} of $G_{\alpha,\beta}^\mu(t)$ is meaningful.

Next, we split the path $\gamma$ into two parts $\gamma_1$ and $\gamma_2$ whose supports $\Gamma_1$ and $\Gamma_2$ are the sets formed by the complex numbers $\lambda$ such that $\lambda^{\mu+1}$ lies on the intersection of $\Gamma$ and the boundaries of $\frac{1}{g(t)}+\Sigma_\theta$ and $S_\phi$ respectively, that is,
\begin{align*}
\Gamma_1=\Gamma\cap\overline{\left\{\frac{1}{g(t)}+\Sigma_\theta\right\}} \quad \mbox{and} \quad \Gamma_2=\Gamma\cap\overline{\left\{S_\phi\right\}}.
\end{align*}
Thus $\Gamma=\Gamma_1\cup\Gamma_2$ and $G_{\alpha,\beta}^\mu(t)=I_1(t)+I_2(t),$ for $t>0,$ where
\begin{align*}
I_j(t):=\frac{1}{2\pi i}\int_{\gamma_j} e^{\lambda t}\frac{\lambda^\mu}{\lambda^\mu+\alpha}\left(\frac{(\lambda-\beta)\lambda^\mu}{\lambda^\mu+\alpha}I-A\right)^{-1}d\lambda, \quad j=1,2.
\end{align*}
Now, we estimate the norms $\|I_1(t)\|$ and $\|I_2(t)\|.$ 

Using \eqref{eq1.1-P}  and \eqref{eq1.2-P} we get that for every $t>0$,
\begin{align}\label{I1-1}
\|I_1(t)\|&\leq\frac{1}{2\pi}\int_{\gamma_1}|e^{\lambda t}|\left|\frac{\lambda^\mu}{\lambda^{\mu+1}-(\beta+\omega)\lambda^{\mu}-\alpha\omega}\right||d\lambda|.
\end{align}
Let $\lambda_{{\min}}$ be the complex $\lambda\in \mathbb{C}$ such that $\operatorname{Im}(\lambda)>0,$ and $|\lambda^{\mu+1}-(\beta+\omega)\lambda^\mu-\alpha\omega|=\operatorname{dist} (L,\alpha\omega),$ where $L$ in the line passing through the point $(\frac{1}{g(t)},0)$ and the intersection of $\Gamma_1$ and $\Gamma_2$.  Then for $\lambda\in \Gamma_1$ we have that

\begin{align}\label{I1-2}
\frac{1}{|\lambda^{\mu+1}-(\beta+\omega)\lambda^{\mu}-\alpha\omega|}
\leq \frac{g(t)}{\cos(\theta)(1+\alpha|\omega|g(t))}, \quad t>0.
\end{align}
It follows from \eqref{I1-1} and \eqref{I1-2} that for every $t>0$,
\begin{align}\label{I1-3}
\|I_1(t)\|\leq\frac{M}{2\pi}\frac{g(t)}{\cos(\theta)(1+\alpha|\omega| g(t))}\int_{\gamma_1}|e^{\lambda t}||\lambda|^{\mu} |d\lambda|.
\end{align}
Recall that
\begin{align*}
|\arg(\lambda^{\mu+1})|<\phi<\frac{\pi}{2}+\theta\le (1+\widetilde \mu)\frac{\pi}{2}.
\end{align*}
Therefore the contour can be chosen such that
\begin{align*}
(\mu+1)\frac{\pi}{2}<|\arg(\lambda^{\mu+1})|<\phi<\frac{\pi}{2}+\theta\le (1+\widetilde \mu)\frac{\pi}{2}.
\end{align*}
Thus letting $\varphi=|\arg(\lambda)|$ we have that $\frac{\pi}{2}<\varphi<\frac{1+\widetilde\mu}{1+\mu}\frac{\pi}{2}$.
It follows from \eqref{I1-3} that (note that $\cos(\varphi)<0$) for every $t>0$ we have that
\begin{align}\label{eq-0-P}
\|I_1(t)\|\leq \frac{M}{2\pi}\frac{g(t)}{\cos(\theta)(1+\alpha|\omega| g(t))}\int_{0}^{\infty}e^{s\cos (\varphi)t} s^\mu ds
\leq\frac{C}{1+\alpha|\omega| g(t)}.
\end{align}

Similarly, if $\lambda\in \Gamma_2$, then for every $t>0$, we have that
\begin{align*}
\frac{1}{|\lambda^{\mu+1}-(\beta+\omega)\lambda^{\mu}-\alpha\omega|}\le \frac{1}{|\lambda^{\mu+1}-\alpha\omega|}\leq \frac{g(t)}{|\cos(\varphi)|}.
\end{align*}
Hence there exists a constant $C>0$ such that for every $t>0$,
\begin{align}\label{eq-1-P}
\|I_2(t)\|&\leq\frac{M}{2\pi}\frac{g(t)}{|\cos(\varphi)|}\int_{\gamma_2}|e^{\lambda t}||\lambda|^{\mu} |d\lambda|\notag\\
&\leq Cg(t)\int_{0}^{\infty}e^{s\cos (\varphi)t} s^\mu ds=C.
\end{align}
It follows from \eqref{eq-0-P} and \eqref{eq-1-P}  that there exists a constant $C>0$  such that
\begin{align}\label{I1-4}
\|G_{\alpha,\beta}^\mu(t)\|\leq C\left[1+  \frac{1}{1+\alpha|\omega|g(t)}\right] ,\;\;\forall \;t>0.
\end{align}
The estimate \eqref{eq-asymp1-P} follows from \eqref{uniq-lp}, \eqref{I1-4} and the strong continuity of the resolvent family on $[0,\infty)$. The proof of part (a) is complete.

(b) Now assume that $\alpha<0$ and $\alpha+\beta^\mu\ge |\alpha|$. We proceed similarly as in part (a) by considering again the same function $g(t)=t^{\mu+1}$, $t>0$.
But here we set
\begin{align*}
\Gamma_1=\Gamma\cap\overline{\left\{\alpha\omega+\frac{1}{g(t)}+\Sigma_\theta\right\}} \quad \mbox{and} \quad \Gamma_2=\Gamma\cap\overline{\left\{S_\phi\right\}}.
\end{align*}
Recall that here $\alpha\omega>0$.
The hypothesis implies that the representation \eqref{eq1.2-P} of $G_{\alpha,\beta}^\mu(t)$ is also valid in this case.
On $\Gamma_1$,  we have that
\begin{align}\label{B1}
\left|\frac{1}{\lambda^{\mu+1}-(\omega+\beta)\lambda^\mu-\alpha\omega}\right|\le \frac{1}{|\lambda^{\mu+1}-\alpha\omega|}\le \frac{g(t)}{\cos(\theta)},\;\;t>0,
\end{align}
and there exists a constant $C>0$ such that
\begin{align}\label{B2}
|\lambda|^{\mu+1}\le  C\left(\alpha\omega+\frac{1}{g(t)}\right).
\end{align}
It follows from \eqref{B2} that on $\Gamma_1$, we have that for every $t>0$,
\begin{align}\label{B3}
|e^{t\lambda}|\le e^{t|\lambda|}\le e^{Ct\left(\alpha\omega+\frac{1}{t^{\mu+1}}\right)^{\frac{1}{\mu+1}}}\le Ce^{(\alpha\omega)^{\frac{1}{\mu+1}}t}.
\end{align}
Using \eqref{B1}, \eqref{B2} and \eqref{B3} we get that for every $t>0$,

\begin{align}\label{B4}
\|I_1(t)\|\le\frac{C}{2\pi}\frac{ g(t)}{\cos(\theta)}\left(\alpha\omega+\frac{1}{g(t)}\right)e^{(\alpha\omega)^{\frac{1}{\mu+1}}t}\int_{\gamma_1}\frac{1}{|\lambda|}\;|d\lambda|.
\end{align}
Since
\begin{align*}
\mbox{length}(\gamma_1)\le C\left(\alpha\omega+\frac{1}{g(t)}\right)^{\frac{1}{\mu+1}},\;\;t>0
\end{align*}
and on $\Gamma_1$ (by using the law of sines),
\begin{align*}
|\lambda|\ge \left[\cos(\theta)\left(\alpha\omega+\frac{1}{g(t)}\right)\right]^{\frac{1}{\mu+1}},\;\;t>0,
\end{align*}
it follows from \eqref{B4} that for every $t>0$,
\begin{align}\label{E1}
\|I_1(t)\|\le C \Big[1+\alpha\omega g(t)\Big]e^{(\alpha\omega)^{\frac{1}{\mu+1}}t}.
\end{align}

Next, we consider the integral $I_2(t)$. Let $z_t$ and $\bar z_t$ be the intersection points of the boundary of $\alpha\omega+\frac{1}{g(t)}+\Sigma_{\theta}$ and $S_{\phi}$, for which we have
\begin{align*}
|\lambda^{\mu+1}-(\omega+\beta)\lambda^\mu-\alpha\omega|\ge |\lambda^{\mu+1}-\alpha\omega|\ge |\cos(\phi)||z_t|,\;\;\lambda\in \Gamma_2,
\end{align*}
and
\begin{align*}
|z_t|\ge C\left(\alpha\omega+\frac{1}{g(t)}\right),\;\;t>0.
\end{align*}
The same bounds are also valid for the conjugate $\bar z_t$ of $z_t$. Letting $\gamma_2(s)=se^{i\varphi}$, $s>0$ (recall that $\cos(\varphi)<0$), we get that for every $t>0$,
\begin{align}\label{E2}
\|I_2(t)\|\leq&\frac{1}{2\pi}\int_{\gamma_2}|e^{\lambda t}|\left|\frac{\lambda^{\mu}}{\lambda^{\mu+1}-(\omega+\beta)\lambda^\mu-\alpha\omega}\right||d\lambda|\notag\\
\le &\frac{1}{2\pi\sin(\phi)|z_t|}\int_{\gamma_2}|e^{\lambda t}|\lambda^{\mu}|\;|d\lambda|\notag\\
\le &\frac{C t^{\mu+1}}{1+\alpha\omega g(t)}\int_0^\infty e^{t\cos(\varphi)s}s^{\mu}\;ds\notag\\
\le &\frac{C}{1+\alpha\omega g(t)}.
\end{align}
It follows from \eqref{E1} and \eqref{E2} that for every $t> 0$,
\begin{align}\label{E3}
\|G_{\alpha,\beta}^\mu(t)\|\le  C \left[1+\alpha\omega g(t)\right]e^{(\alpha\omega)^{\frac{1}{\mu+1}}t}.
\end{align}
Now the estimate \eqref{eq-asymp2-P} follows from \eqref{uniq-lp}, \eqref{E3} and the strong continuity of the resolvent family on $[0,\infty)$. The proof is finished.
\end{proof}

\section{Explicit representation of resolvent families}\label{sec-par-ca}

In this section, we consider particular examples of the operator $A$ and give a more explicit (than the one in \eqref{uni-sol}) representation of the resolvent family associated with the system \eqref{eq-IV-0} in terms of some special functions and we also investigate their precise exponential bound.

\subsection{The case where  $A=\rho I$}

Here, we assume that the operator $A$ is given by $A=\rho I$ for some $\rho\in\R$.
Hence, the system \eqref{eq-IV-0} becomes

\begin{equation}\label{escal}
\begin{cases}
\displaystyle u'(t)=\rho u(t)+\frac{\rho\alpha}{\Gamma(\mu)}\int_{0}^t e^{-\beta(t-s)}(t-s)^{\mu-1}u(s)ds ,\;\;\;\, t>0,\\
u(0)=u_0.
\end{cases}
\end{equation}
We have the following explicit representation of solutions.

\begin{proposition}\label{prop-scar}
Let $\alpha\in\mathbb R$, $\alpha\ne 0$, $0< \mu\le1,\beta\ge 0$ and $A=\rho I$ for some $\rho\in\R$.
Assume that $\alpha> 0$ or $\alpha<0$ and $\alpha+\beta^\mu\ge |\alpha|$.
Then the strongly continuous exponentially bounded resolvent family $S_{\alpha,\beta}^\mu$ associated with the system \eqref{escal} is given by
\begin{align}\label{S-scal}
S_{\alpha,\beta}^\mu(t)=e^{-\beta t}\sum_{k=0}^\infty(\rho+\beta)^kt^k E_{\mu+1,k+1}^{(k+1)}(\alpha\rho t^{\mu+1}),\;\;t\ge 0,
\end{align}
provided that the series converges and where
\begin{align}\label{MLF}
E_{\mu+1,k+1}^{(k+1)}(z):=\sum_{n=0}^\infty \frac{(k+n)!z^n}{n!k!\Gamma(n\mu+k+n+1)}, \;\;\;z\in\mathbb{C},
\end{align}
denotes the generalized Mittag-Leffler function.
\end{proposition}

\begin{proof}
Let $\alpha\in\mathbb R$, $\alpha\ne 0$, $0<\mu\le 1$, $\beta\ge 0$ and $A=\rho I$ for some $\rho\in\R$.
Assume that $\alpha> 0$ or $\alpha<0$ and $\alpha+\beta^\mu\ge |\alpha|$. Then it follows from Corollary \ref{cor-generation} that there exists a strongly continuous resolvent family $S_{\alpha,\beta}^\mu$ of type $(M,\widetilde \omega)$ (for some $\widetilde\omega>0$) such that the unique solution of \eqref{escal} is given by \eqref{uni-sol}. In addition we have that the Laplace transform of $S_{\alpha,\beta}^\mu(t)$ is given by
\begin{align}\label{Lap-S}
\widehat S_{\alpha,\beta}^\mu(\lambda):=\frac{(\lambda+\beta)^{\mu}}{(\lambda+\beta)^{\mu+1}-(\rho+\beta)(\lambda+\beta)^\mu-\alpha\rho},\;\;\op(\lambda)>\widetilde\omega.
\end{align}
Using the properties of Laplace transform, we have that
\begin{align}\label{M1}
S_{\alpha,\beta}^\mu(t)=e^{-\beta t}G_{\alpha,\beta}^\mu(t),
\end{align}
for some function $G_{\alpha,\beta}^\mu(t)$ whose Laplace transform is given for all $\op(\lambda)>\widetilde\omega+\beta$ by
\begin{align}\label{Lap-G}
\widehat G_{\alpha,\beta}^\mu(\lambda)=\frac{\lambda^{\mu}}{\lambda^{\mu+1}-(\rho+\beta)\lambda^\mu-\alpha\rho}.
\end{align}
Since
\begin{align*}
\left|\frac{(\rho+\beta)\lambda^\mu}{\lambda^{\mu+1}-\alpha\rho}\right|<1,
\end{align*}
for $\lambda$ large enough, we have that
\begin{align*}
\widehat G_{\alpha,\beta}^\mu(\lambda)=&\frac{\lambda^{\mu}}{\lambda^{\mu+1}-(\rho+\beta)\lambda^\mu-\alpha\rho}
=\frac{\lambda^{\mu}} {[\lambda^{\mu+1}-\alpha\rho][1-\frac{(\rho+\beta)\lambda^\mu}{\lambda^{\mu+1}-\alpha\rho}]}\notag\\
=&\frac{\lambda^{\mu}} {\lambda^{\mu+1}-\alpha\rho}\sum_{k=0}^\infty \frac{(\rho+\beta)^k\lambda^{k\mu}}{[\lambda^{\mu+1}-\alpha\rho]^k}
=\sum_{k=0}^\infty \frac{(\rho+\beta)^k\lambda^{(k+1)\mu}}{[\lambda^{\mu+1}-\alpha\rho]^{k+1}}\notag\\
=&\sum_{k=0}^\infty \frac{(\rho+\beta)^{k}\lambda^{2(k+1)\mu}}{\lambda^{(k+1)\mu}[\lambda^{\mu+1}-\alpha\rho]^{k+1}}.
\end{align*}
Taking the inverse Laplace transform, we get that (see e.g. \cite[Formula 17.6]{Ha-Ma-Sa-11})
\begin{align*}
\mathcal L^{-1}\left(\frac{\lambda^{2(k+1)\mu}}{\lambda^{(k+1)\mu}[\lambda^{\mu+1}-\alpha\rho]^{k+1}}\right)(t)=t^k E_{\mu+1,k+1}^{(k+1)}(\alpha\rho t^{\mu+1}),
\end{align*}
where $ E_{\mu+1,k+1}^{(k+1)}$ is the generalized Mittag-Leffler function given in \eqref{MLF}. We have shown that
\begin{align}\label{M2}
G_{\alpha,\beta}^\mu(t)=\sum_{k=0}^\infty (\rho+\beta)^kt^k E_{\mu+1,k+1}^{(k+1)}(\alpha\rho t^{\mu+1}).
\end{align}
Now \eqref{S-scal} follows from \eqref{M1} and \eqref{M2}. The proof is finished.
\end{proof}

We have the following result as a direct consequence of Theorem \ref{th2-P}.

\begin{corollary}
Let $\alpha\in\mathbb R$, $\alpha\ne 0$, $0<\mu\le 1$, $\beta\ge 0$, $A=\rho I$ for some $\rho< 0$ and assume in addition that $\rho+\beta\le 0$. Then the following assertions hold.
\begin{enumerate}
\item If $\alpha>0$, then there exists a constant $M>0$ such that
\begin{align}\label{EST1}
\|S_{\alpha,\beta}^\mu(t)\|\le Me^{-\beta t},\;\;\;\forall\;t\ge 0.
\end{align}
\item If $\alpha<0$ and $\alpha+\beta^\mu\ge |\alpha|$, then there exists a constant $M>0$ such that
\begin{align}\label{EST2}
\|S_{\alpha,\beta}^\mu(t)\|\le M\left(1+\alpha\rho t^{\mu+1}\right)e^{-(\beta-(\alpha\rho)^{\frac{1}{\mu+1}}) t},\;\;\;\forall\;t\ge 0.
\end{align}
\end{enumerate}
\end{corollary}

\begin{proof}
This is a particular case of Theorem \ref{th2-P} since  $A$ is $\omega$-sectorial of angle $\theta$ for every $0<\theta<\frac{\pi}{2}$, with $\omega=\rho<0$ and $\omega+\beta=\rho+\beta\le 0$ by assumption.
\end{proof}

In the following remark we show that if $\mu=1$ one recovers some of the results contained in the references \cite{Ch-Xi-Li-09,Li-Po-11}.

\begin{remark}
Let $\mu=1$ and $A=\rho I$ for some $\rho<0$. It follows from \eqref{Lap-G} that
\begin{align*}
\widehat G_{\alpha,\beta}^1(\lambda)=\frac{\lambda}{\lambda^2-(\rho+\beta)\lambda-\alpha\rho}.
\end{align*}
Let
\begin{align*}
D:=(\rho+\beta)^2+4\alpha\rho.
\end{align*}
We have the following three situations.
\begin{itemize}
\item[(i)] If $D>0$, then let $\lambda_1:=\frac{\rho+\beta+\sqrt D}{2}$ and $\lambda_2:=\frac{\rho+\beta-\sqrt D}{2}$. Using partial fractions, we get that
\begin{align*}
\widehat G_{\alpha,\beta}^1(\lambda)=\frac{1}{\sqrt D}\left(\frac{\lambda_1}{\lambda-\lambda_1}-\frac{\lambda_2}{\lambda-\lambda_2}\right).
\end{align*}
This implies that
\begin{align*}
G_{\alpha,\beta}^1(t)=\frac{\rho+\beta+\sqrt D}{2\sqrt D}e^{\frac{\rho+\beta+\sqrt D}{2} t}-\frac{\rho+\beta-\sqrt D}{2\sqrt D}e^{\frac{\rho+\beta-\sqrt D}{2} t},\;\;\forall\;t\ge 0,
\end{align*}
so that
\begin{align*}
S_{\alpha,\beta}^1(t)=\frac{\rho+\beta+\sqrt D}{2\sqrt D}e^{\frac{\rho-\beta+\sqrt D}{2} t}-\frac{\rho+\beta-\sqrt D}{2\sqrt D}e^{\frac{\rho-\beta-\sqrt D}{2} t},\;\forall\;t\ge 0.
\end{align*}
In that case one has uniform exponential stability if and only if
\begin{align*}
\rho-\beta+\sqrt{(\rho+\beta)^2+4\alpha\rho}<0.
\end{align*}
\item[(ii)] If $D=0$, then
\begin{align*}
\widehat G_{\alpha,\beta}^1(\lambda)=\frac{\lambda}{\left(\lambda-\frac{\rho+\beta}{2}\right)^2}=\frac{1}{\lambda-\frac{\rho+\beta}{2}}+\frac{\rho+\beta}{2}\frac{1}{\left(\lambda-\frac{\rho+\beta}{2}\right)^2}.
\end{align*}
This implies that
\begin{align*}
G_{\alpha,\beta}^1(t)=\left(1+\frac{\rho+\beta}{2}t\right)e^{\frac{\rho+\beta}{2} t},\;\;\forall\;t\ge 0,
\end{align*}
so that
\begin{align*}
S_{\alpha,\beta}^1(t)=\left(1+\frac{\rho+\beta}{2}t\right)e^{\frac{\rho-\beta}{2} t},\;\;\forall\;t\ge 0.
\end{align*}
In that case one has uniform exponential stability if and only if
\begin{align*}
\rho-\beta<0.
\end{align*}
\item[(iii)] Now if $D<0$, then let $\lambda_0:=\frac{\rho+\beta +i\sqrt{-D}}{2}$ and $\overline{\lambda_0}=\frac{\rho+\beta -i\sqrt{-D}}{2}$. Proceeding as above we get that for every $t\ge 0$,
\begin{align*}
S_{\alpha,\beta}^1(t)=\left(\cos\left(\sqrt{-D}t\right)-i\frac{\rho+\beta}{\sqrt{-D}}\sin\left(\sqrt{-D}t\right)\right)e^{\frac{\rho-\beta}{2} t},
\end{align*}
so that one has uniform exponential stability if and only if
\begin{align*}
\rho-\beta<0.
\end{align*}
\end{itemize}
\end{remark}

\subsection{The case of a self-adjoint operator}

We assume that $-A$ is a non-negative and self-adjoint operator with compact resolvent on the Hilbert space $L^2(\Omega)$ where $\Omega\subset\mathbb R^N$ is a bounded open set.
Since $-A$ is non-negative and self-adjoint  with compact resolvent, then it has a discrete spectrum which is formed of eigenvalues. Its eigenvalues satisfy $0\le\lambda_1\le\lambda_2\le\cdots\le\lambda_n\le\cdots$ and $\lim_{n\to\infty}\lambda_n=\infty$. We denote the normalized eigenfunction associated with $\lambda_n$ by $\phi_n$. Then $\{\phi_n : n\in \mathbb{N} \}$ is an orthonormal basis for $L^2(\Omega)$ and the set is also total in $D(A)$.
Observe also that for every $v\in D(A)$ we can write
\begin{align*}
-Av=\sum_{k=1}^\infty\lambda_n \langle v,\phi_n\rangle_{L^2(\Omega)} \phi_n.
\end{align*}

 We have the following result as a direct consequence of Proposition \ref{prop-scar}.

 \begin{corollary}
Let $\alpha\in\mathbb R$, $\alpha\ne 0$, $0< \mu\le1$, $\beta>0$ and let $A$ be as above.
Assume that $\alpha> 0$ or $\alpha<0$ and $\alpha+\beta^\mu\ge |\alpha|$.
Then the strongly continuous exponentially bounded resolvent family $S_{\alpha,\beta}^\mu$ associated with the system \eqref{eq-IV-0} is given by
 \begin{align}\label{sol-hil-2}
S_{\alpha,\beta}^\mu(t)=e^{-\beta t}\sum_{n=1}^\infty\sum_{k=0}^\infty(\beta-\lambda_n)^kt^k E_{\mu+1,k+1}^{(k+1)}(-\alpha\lambda_nt^{\mu+1}),\;\;t\ge 0,
\end{align}
provided that the series converges.
 \end{corollary}

 \begin{proof}
 Let $A$ be as above.
Multiplying both sides of \eqref{eq-IV-0} by $\phi_n(x)$ and integrating over the set $\Omega$  we get that for every $n\in\mathbb N$, the function $u_n(t):=\langle u(t),\phi_n\rangle_{L^2(\Omega)}$ is a solution of the system
\begin{align*}
\begin{cases}
\displaystyle u_n'(t)=-\lambda_n u_n(t)+\frac{-\lambda_n\alpha}{\Gamma(\mu)}\int_{0}^t e^{-\beta(t-s)}(t-s)^{\mu-1}u_n(s)ds, \,\; t>0\\
u_n(0)=u_{0,n},
\end{cases}
\end{align*}
where $u_{0,n}=\langle u_0,\phi_n\rangle_{L^2(\Omega)}$.
It follows from Proposition \ref{prop-scar} that
\begin{align*}
u_n(t)=S_{\alpha,\beta,n}^\mu(t)u_{0,n},\;\forall\;t\ge 0,
\end{align*}
where for every $n\in\N$,
\begin{align*}
S_{\alpha,\beta,n}^\mu(t)=e^{-\beta t}\sum_{k=0}^\infty(\beta-\lambda_n)^kt^k E_{\mu+1,k+1}^{(k+1)}(-\alpha\lambda_nt^{\mu+1}),\;\forall\;t\ge 0.
\end{align*}
Since
\begin{align*}
u(t,x)=\sum_{n=1}^\infty u_n(t)\phi_n(x),\;\forall\;t\ge 0,\;x\in\Omega,
\end{align*}
we have that
\begin{align*}
u(t,x)=\sum_{n=1}^\infty S_{\alpha,\beta,n}^\mu(t)u_{0,n}\phi_n(x),\;\;\forall\;t\ge 0,\;x\in\Omega,
\end{align*}
so that $S_{\alpha,\beta}^\mu$ is given by the expression in \eqref{sol-hil-2}. The proof is finished.
 \end{proof}

\begin{corollary}
Let $\alpha\in\mathbb R$, $\alpha\ne 0$, $0< \mu\le 1$, $\beta\ge 0$ and let $A$ be as above. Assume that the first eingenvalue $\lambda_1>0$  and that $\beta-\lambda_1\le 0$. Then the following assertions hold.
\begin{enumerate}
\item If $\alpha>0$, then there exists a constant $M>0$ such that
\begin{align*}
\|S_{\alpha,\beta}^\mu(t)\|\le Me^{-\beta t},\;\;\;\forall\;t\ge 0.
\end{align*}
\item If $\alpha<0$ and $\alpha+\beta^\mu\ge |\alpha|$, then there exists a constant $M>0$ such that
\begin{align*}
\|S_{\alpha,\beta}^\mu(t)\|\le M\left(1-\alpha\lambda_1t^{\mu+1}\right)e^{-(\beta-(-\alpha\lambda_1)^{\frac{1}{\mu+1}}) t},\;\;\;\forall\;t\ge 0.
\end{align*}
\end{enumerate}
\end{corollary}

\begin{proof}
This is a particular case of  Theorem \ref{th2-P}, given that by assumption, the operator $A$ is $\omega$-sectorial of angle $\theta$ for every $0< \theta<\frac{\pi}{2}$, with $\omega=-\lambda_1<0$ and that $\beta+\omega=\beta-\lambda_1\le 0$.
\end{proof}




\end{document}